\documentclass{amsart}

\usepackage{amsmath,amsthm,amssymb}
\usepackage{mathrsfs}
\usepackage{cite}

\newcommand{\es}{{\mathcal{S}}}
\newcommand{\IN}{{\mathbb N}}
\newcommand{\ID}{{\mathbb D}}
\newcommand{\IC}{{\mathbb C}}
\newcommand{\A}{{\mathcal{A}}}
\newcommand{\U}{{\mathcal{U}}}

\newtheorem{theorem}{Theorem}[section]
\newtheorem{lemma}{Lemma}[section]

\newtheorem{remark}{Remark}[section]

\numberwithin{equation}{section} \numberwithin{theorem}{section}

\keywords{analytic, class $\mathcal{U}(\lambda)$, Zalcman conjecture, generalised Zalcman conjecture, Krushkal inequality, Hankel determinant}

\begin{document}

\title[On certain properties of the class $\mathcal{U}(\lambda)$ ]{On certain properties of the class $\boldsymbol{\mathcal{U}(\lambda)}$}

\author[N. M. Alarifi]{Najla M. Alarifi}
\address{Department of Mathematics, Imam Abdulrahman Bin Faisal University, Dammam 31113, Kingdom of Saudi Arabia}
\email{najarifi@gmail.com}

\author[M. Obradovi\'c] {Milutin. Obradovi\'c}
\address{Department of Mathematics, Faculty of Civil Engineering, University of Belgrade, Bulevar Kralja Aleksandra 73,11000 Belgrade, Serbia}
\email{obrad@grf.bg.ac.rs}

\author[N. Tuneski] {Nikola. Tuneski}
\address{Department of Mathematics and Informatics, Faculty of Mechanical Engineering,
Ss. Cyril and Methodius University in Skopje, Karpo\'s II b.b., 1000 Skopje, Republic of
North Macedonia}
\email{nikola.tuneski@mf.edu.mk}

\begin{abstract}
Let ${\mathcal A}$ be the class of functions analytic  in the unit disk $\ID := \{ z\in \IC:\, |z| < 1 \}$ and  normalized such that $f(z)=z+a_2z^2+a_3z^3+\cdots$. In this paper we study the class $\mathcal{U}(\lambda)$, $0<\lambda \leq1$, consisting of functions $f$ from $\A$ satisfying
\[\left|\left(\frac{z}{f(z)}\right)^2f'(z)-1\right| < \lambda \quad (z\in\ID)\]
and give results regarding the Zalcman Conjecture, the generalised Zalcman conjecture, the Krushkal inequality and the second and third order Hankel determinant.
\end{abstract}

\subjclass[2010]{Primary: 30C45, 30C50; Secondary: 30C80}

\maketitle

\section{Introduction and Preliminaries}

Let $\A$ be the class of functions that are analytic in the open unit disk $\ID=\{z:|z|<1\}$ and normalised such that $f(0)=f'(0)-1=0$, i.e., have expansion $f(z)=z+a_2z^2+a_3z^3+\cdots$.

\medskip

Further, let
\[\U(\lambda) = \left\{ f\in\A: \left|U_f(z)-1\right| < \lambda, z\in\ID  \right\},\]
where $0<\lambda\le1$ and
\[U_f(z):=\left(\frac{z}{f(z)}\right)^2f'(z).\]
The functions from $\U(\lambda)$ are univalent, its special case when $\lambda=1$ first studied in \cite{Japonica_1996}  and more details on them can be found in \cite{obpon-1,obpon-2,TTV}).

\medskip

An intriguing fact about $\U(\lambda)$ is that in spite the class $\mathcal{S}^\ast$ of starlike functions is very large and contains most classes of univalent functions, it doesn't contain the class $\U\equiv \U(1)$, i.e., $\U$ is not in $\mathcal{S}^\ast$, nor vice versa.
Namely, the function $-\ln(1-z)$ is convex, thus starlike, but not in $\mathcal{U}$ because $\operatorname U_f(0.99)=3.621\ldots>1$, while  the function $f$ defined by $\frac{z}{f(z)}=1-\frac32 z + \frac12z^{3}=(1- z)^{2} \left(1+\frac{z}{2}\right)$ is in $\mathcal{U}$ and such that $\frac{zf'(z)}{f(z)} = -\frac{2 \left(z^2+z+1\right)}{z^2+z-2} = -\frac{1}{5}+\frac{3 i}{5}$ for $z=i$. This rear property is the main reason why the class $\mathcal{U}$ (and $\mathcal{U}(\lambda)$) attracts huge attention in the past decades.

\medskip

In this paper we study class $\mathcal{U}(\lambda)$ regarding the Zalcman Conjecture, the generalised Zalcman conjecture, the Krushkal inequality and the second and third order Hankel determinant which will be defined in corresponding sections further in the paper.

\medskip

For the study we will need the following result proven in \cite{OP-2019} as a part of the proof of Theorem 1.

 \begin{lemma}\label{lem1}
 For each function $f$ in $\mathcal{U}(\lambda),$ $0< \lambda \leq 1,$  there exists function $\omega _1,$ analytic in $\mathbb{D},$ such that $|\omega _1(z)|\leq |z| < 1,$ and $|\omega '_1(z)|\leq 1,$
 for all $z \in \mathbb{D},$ with
 \begin{equation}\label{eq1.1}
 \frac{z}{f(z)} =1-a_2z-\lambda z\omega _1(z).
\end{equation}
 Additionally, for $\omega _1(z)=c_1z+c_2z^2+\cdots,$
 \begin{equation}\label{eq1.2}
 |c_1|\leq 1,\quad |c_2|\leq \frac{1}{2}(1-|c_1|^2)\quad   and \quad   |c_3|\leq \frac{1}{3}\left[1-|c_1|^2-\frac{4|c_2|^2}{1+|c_1|}\right].
 \end{equation}
\end{lemma}

\medskip

Using \eqref{eq1.1}, we have
  \begin{equation*}
z =[1-a_2z-\lambda z\omega _1(z)]f(z),
\end{equation*}
and after equating the coefficients,
\begin{align}\label{eq1.3}
&a_3=\lambda c_1+a^2_2,\nonumber\\
&a_4=\lambda c_2+2\lambda a_2c_1+a_2^3,\\
\nonumber&a_5=\lambda c_3+2\lambda a_2c_2+\lambda ^2c_1^2+3\lambda a_2^2c_1+a_2^4,
\end{align}
that we will use later on.

\medskip

We will also use the next result (\cite{OP-2016}).

\begin{lemma}\label{lem2}
Let $f\in\U(\lambda)$ for  $0<\lambda\le1$, and be given by $f(z)=z+\sum_{n=2}^{\infty}a_{n}z^n$. Then
 \begin{equation}\label{eq-9}
  |a_2|\leq 1+\lambda.
 \end{equation}
 If $|a_2|=1+\lambda$, then $f$ must be of the form
\begin{equation}\label{ch-9-U-eq-003}
f(z) = \frac{z}{1-(1+\lambda)e^{i\phi}z+\lambda e^{2i\phi}z^2}
\end{equation}
for some $\phi\in[0,2\pi]$.
\end{lemma}

\medskip

In the same paper (\cite{OP-2016}) it was conjectured that for functions in $\U(\lambda)$, $|a_n|\le\sum_{i=0}^{n-1}\lambda^i$ holds sharply, and was claimed to be proven in the case $n=3$:
\begin{equation}\label{eq-9-2}
|a_3|\leq 1+\lambda+\lambda^2.
\end{equation}
The proof rely on another claim, that for all functions $f$ from $\U(\lambda)$,
\begin{equation}\label{subord}
 \frac{f(z)}{z}\prec \frac{1}{(1+z)(1+\lambda z)}.
 \end{equation}
Recently, in \cite{lipon}, the second claim, and consequently the first one also, was proven to be wrong by giving a counterexample. Still, the subset of $\U(\lambda)$ when the inequality \eqref{eq-9-2} and subordination \eqref{subord} hold is nonempty, as the function
\begin{equation}\label{eq1.5}
f_\lambda (z)= \frac{z}{(1-z)(1-\lambda z)}=\sum_{n=1}^{\infty}\frac{1-\lambda^n}{ 1-\lambda }z^n=z+(1+\lambda)z^2+(1+\lambda+\lambda^2)z^2+\cdots
\end{equation}
shows. Here  $\left.\frac{1-\lambda^n}{ 1-\lambda }\right|_{\lambda =1}=n$ for all $n=1,2,3,\ldots$.

\medskip

Let note that from \eqref{eq1.3} and  \eqref{eq-9-2},
   \begin{equation*}
|a_3|=|\lambda c_1+a^2_2|\leq 1+\lambda+\lambda^2,
\end{equation*}
i.e.,
   \begin{equation}\label{eq1.6}
 |\lambda c_1+a^2_2|\leq 1+\lambda+\lambda^2,
\end{equation}
 which we will use further in the proofs.

\medskip

 \section{Zalcman and generalised Zalcman conjecture for the class $\mathcal{U}(\lambda)$}

In 1960 Zalcman posed the conjecture:
\[ |a_n^2-a_{2n-1}|\le (n-1)^2 \quad\quad (n\in \IN, n\ge2), \]
proven in 2014 by Krushkal (\cite{krushkal}) for the whole class $\es$ by using complex geometry of the universal Teichm\"{u}ller space. In 1999, Ma (\cite{ma}) proposed a generalized Zalcman conjecture,
\[ |a_m a_n-a_{m+n-1}|\le (m-1)(n-1) \quad\quad (m,n\in \IN, m\ge2, n\ge2),\]
which is still an open problem, closed by Ma  for the class of starlike functions and for the class of univalent functions with real coefficients. Ravichandran and Verma in \cite{ravi} closed it for the classes of starlike and convex functions of given order and for the class of functions with bounded turning.

\medskip

  In \cite{OP-2020}, the authors treated some particulars cases of those problems and obtained sharp results. Since $\mathcal{U}(\lambda)\subseteq \mathcal{U}$ for $0< \lambda \leq 1,$
the same results are valid for the class $\mathcal{U}(\lambda),$ but those results are not sharp for the class $\mathcal{U}(\lambda).$ In this part of the paper we find better and sharp results for the functions $f$ from  $\mathcal{U}(\lambda)$ satisfuing subordination \eqref{subord}.

\begin{theorem}\label{thm1}
Let $f$ in $\mathcal{U}(\lambda)$ be of the form $f(z)=z+a_2z^2+a_3z^3+\cdots$. Then
\begin{itemize}
\item[($i$)] $|a^2_2-a_3|\leq \lambda ,$ $0< \lambda \leq 1$;
\item[($ii$)] $|a^2_3-a_5|\leq \lambda (1+\lambda)^2,$ $\frac{1}{2}\left(\frac{\sqrt{23}}{3}-1\right)\leq \lambda \leq 1$, if the third coefficient of $f$ satisfies inequality \eqref{eq-9-2}.
\end{itemize}
Both results are sharp as the function $f_\lambda$ given by \eqref{eq1.5} shows.
\end{theorem}

 \begin{proof}$ $\\
($i$) From $a_3=\lambda c_1+a^2_2,$ we have $|a^2_2-a_3|= |-\lambda  c_1| \leq \lambda$.

\medskip

\noindent
($ii$)  Using the relations  \eqref{eq1.2}, \eqref{eq-9} and  \eqref{eq1.6}, we have
\begin{align}\label{eq1.7}
&|a^2_3-a_5|=|\lambda c_3+2\lambda a_2c_2+\lambda a_2^2c_1|\nonumber\\
&\quad =\lambda |  c_3+2  a_2c_2-\lambda c_1^2 +c_1(\lambda c_1+a_2^2)|\nonumber\\
&\quad \leq \lambda (|  c_3|+2  |a_2||c_2|+ \lambda |c_1|^2 +|c_1||\lambda c_1+a_2^2|)\nonumber\\
&\quad \leq \lambda \left[\frac{1}{3}\left(1-|  c_1|^2-  \frac{4|c_2|^2}{1+|c_1|} \right)+2(1+ \lambda) |c_2|+\lambda |c_1|^2+(1+\lambda+\lambda^2)|c_1|\right]\nonumber\\
&\quad =\lambda g_1(|c_1|,|c_2|),
\end{align}
where
\[
g_1(x,y)=\frac{1}{3}\left(1-x^2-  \frac{4y^2}{1+x} \right)+2(1+ \lambda) y+\lambda x^2+(1+\lambda+\lambda^2)x,
\]
$0\leq x=|c_1|\leq 1,$ $ 0\leq y=|c_2|\leq\frac{1}{2}(1-x^2).$

\medskip

Since  $\frac{\partial g_1(x,y)}{\partial x}>0$ for all $(x,y)\in E:=\{(x,y): 0\leq x\leq 1,$  $0\leq y\leq\frac{1}{2}(1-x^2)\} ,$
then $g_1$ has no singular points in the interior of $E$ and $g_1$ attains its maximum on the boundary of  $E.$

\medskip
For $x=0,$ we have $ 0\leq y\leq \frac{1}{2}$ and
\begin{align*}
g_1(0,y)&= \frac{1}{3}(1-4y^2)+2(1+\lambda)y\nonumber\\
&=\frac{1}{3}[-4y^2+6(1+\lambda)y+1]\nonumber\\
&\leq1+\lambda.
\end{align*}

\medskip

Similarly, for  $y=0 $, we have $ 0\leq x\leq 1$ and
\begin{align*}
g_1(x,0)&= \frac{1}{3}(1-x^2)+ \lambda x^2+(1+\lambda+\lambda^2)x\nonumber\\
&=\left(\lambda-\frac{1}{3}\right)x^2+(1+\lambda+\lambda^2)x+\frac{1}{3}\nonumber\\
&\leq (1+\lambda)^2.
\end{align*}

\medskip
Finally, for $ 0\leq x\leq 1$ and $y=\frac{1}{2}(1-x^2) $, we have
   \begin{equation*}
g_1\left(x,\frac{1}{2}(1-x^2)\right)=1+\lambda +\left(\frac{4}{3}+\lambda+\lambda^2\right)x-x^2-\frac{1}{3}x^3=\varphi_1(x).
\end{equation*}

\medskip

Since $\varphi'_1(x)= \frac{4}{3}+\lambda+\lambda^2-2x -x^2\geq \lambda^2+\lambda-\frac{5}{3}\geq 0$ for $\frac{1}{2}\left(\frac{\sqrt{23}}{3}-1\right)\leq \lambda \leq 1,$ then  $\varphi_1$ is an increasing function for $ 0\leq x\leq 1,$ and $\varphi_1(x)\leq (1+\lambda)^2.$

\medskip
Using  \eqref{eq1.7} and all these facts, we have the statement of Theorem \ref{thm1}($ii$).
\end{proof}

 \begin{remark}
For $0<\lambda <\frac{1}{2} \left( \frac{\sqrt{23}}{3}-1\right)=0.8844\ldots$ we have that the function $\varphi_1$ attains its maximum for $x_0=\sqrt{\lambda^2+\lambda+\frac{7}{3}}-1$ and
$\varphi_1(x_0)=\frac{2}{3}\left(  \lambda^2+\lambda+\frac{7}{3}\right)^\frac{3}{2}-1-\lambda^2,$ and so
$|a^2_3-a_5|\leq \lambda\left( \frac{2}{3}\big(  \lambda^2+\lambda+\frac{7}{3}\big)^\frac{3}{2}-1-\lambda^2\right)$.
\end{remark}

\medskip

Next we consider the Generlized Zalcman conjecture.
\medskip

\begin{theorem}\label{thm1}
Let $f$ in $\mathcal{U}(\lambda)$ be of the form $f(z)=z+a_2z^2+a_3z^3+\cdots$. Then
\begin{itemize}
\item[($i$)] $|a_2a_3-a_4|\leq \lambda(\lambda+1) ,$ $ 0< \lambda \leq 1$;
\item[($ii$)] $|a_2a_4-a_5|\leq \lambda(1+\lambda+\lambda^2),$ ${\sqrt\frac{2}{3}}  \leq \lambda \leq 1$, if the third coefficient of $f$ satisfies inequality \eqref{eq-9-2}.
\end{itemize}
Both results are sharp as the function $f_\lambda$ given by \eqref{eq1.5} shows.
\end{theorem}

\begin{proof}$ $\\
($i$) Using \eqref{eq1.2}, \eqref{eq1.3} and \eqref{eq-9} we obtain
\begin{align*}
|a_2 a_3-a_4|&=|\lambda c_2+\lambda a_2c_1|\nonumber\\
&\leq \lambda (|  c_2|+ |a_2||c_1|)\nonumber\\
&\leq \lambda \left[\frac{1}{2}(1-|  c_1|^2)+(1+ \lambda) |c_1|\right]\nonumber\\
&\leq \lambda (1+\lambda).
\end{align*}
\medskip

\noindent($ii)$ Similarly,
\begin{align*}
|a_2 a_4-a_5|&= \lambda| c_3+ a_2c_2+a^2_2 c_1+\lambda c^2_1|\nonumber\\
&\leq \lambda(| c_3|+ |a_2||c_2|+|c_1||a^2_2+\lambda c_1|)\nonumber\\
&\leq \lambda \left[\frac{1}{3}\left(1-|  c_1|^2-  \frac{4|c_2|^2}{1+|c_1|} \right)+(1+ \lambda) |c_2|+(1+\lambda+\lambda^2)|c_1|\right]\nonumber\\
&=\lambda g_2(|c_1|,|c_2|),
\end{align*}
where
\[g_2(x,y)=\frac{1}{3}\left(1-x^2-  \frac{4y^2}{1+x} \right)+(1+ \lambda) y+(1+\lambda+\lambda^2)x,\]
$0\leq x=|c_1|\leq 1,$ $  0\leq y=|c_2|\leq\frac{1}{2}(1-x^2).$

\medskip

Since  $\frac{\partial g_2(x,y)}{\partial x}>0$ for all $(x,y)\in E,$
then the function $g_2$ attains its maximum on the boundary of  $E.$ In that sense, let consider all cases

\medskip

For $x=0,$ we have $ 0\leq y\leq \frac{1}{2}$ and
\begin{align*}
g_2(0,y)&= \frac{1}{3}(1-4y^2)+(1+\lambda)y\nonumber\\
&=\frac{1}{3}[-4y^2+3(1+\lambda)y+1]\nonumber\\
&\leq \frac{1}{2} (1+\lambda),
\end{align*}
since $\lambda\geq \sqrt{\frac{2}{3}}>\frac13.$

\medskip

For $ 0\leq x\leq 1 $ and $y=0 $ we have
\begin{align*}
g_2(x,0)&= \frac{1}{3}(1-x^2)+(1+\lambda+\lambda^2)x\nonumber\\
&= \frac{1}{3}[-x^2+3(1+\lambda+\lambda^2)x+1]\nonumber\\
&\leq 1+\lambda+\lambda^2.
\end{align*}

\medskip
Finally, for $ 0\leq x\leq 1,$ $y=\frac{1}{2}(1-x^2) $ we have
   \begin{equation*}
g_2\left(x,\frac{1}{2}(1-x^2)\right)=\frac{1}{2}(1+\lambda) +\left(\lambda^2+\lambda+\frac{4}{3}\right)x-\frac{1}{2}(1+\lambda)x^2-\frac{1}{3}x^3\equiv \varphi_2(x).
\end{equation*}
Since for $\lambda^2\geq \frac{2}{3},$
 \begin{equation*}
 \varphi'_2(x)=\lambda^2+\lambda+ \frac{4}{3}-(1+\lambda)x -x^2\geq \lambda^2 -\frac{2}{3}\geq 0,
 \end{equation*}
 we realize that
  $\varphi_2$ is an increasing function and
   \begin{equation*}
  \varphi_2(x)\leq \varphi_2(1) = 1+\lambda+ \lambda^2.
  \end{equation*}

  \medskip

  From all the previous facts, we have the statement of the theorem.
\end{proof}

\medskip

 \section{Krushkal inequality for the class $\mathcal{U}(\lambda)$}

The inequality
\[ |a_n^p-a_2^{p(n-1)}|\le 2^{p(n-1)}-n^p\]
was introduced by Krushkal who claims to have proven its sharpness for the whole class of univalent functions in \cite{krushkal}. This claim is not confirmed by the mathemarical community.

\medskip

In this section we give direct proof over the class $\U(\lambda)$ for the cases $n=4$, $p=1$ and $n=5$, $p=1$.

\medskip

\begin{theorem}\label{thm1}
Let $f$ in $\mathcal{U}(\lambda)$ be of the form $f(z)=z+a_2z^2+a_3z^3+\cdots$. Then
\begin{itemize}
\item[($i$)] $|a_4-a^3_2|\leq 2\lambda(\lambda+1) $;
\item[($ii$)]  $|a_5-a^4_2|\leq \lambda(3+5\lambda+3\lambda^2)$, if the third coefficient of $f$ satisfies inequality \eqref{eq-9-2}.
\end{itemize}
Both results are sharp as the function $f_\lambda$ given by \eqref{eq1.5} shows.
\end{theorem}

\begin{proof}$ $\\
($i$) Similarly as in the proofs of two previous theorem, we have
\begin{align*}
|a_4-a^3_2|&= \lambda| c_2+2 a_2c_1|\nonumber\\
&\leq \lambda (|  c_2|+ 2|a_2||c_1|)\nonumber\\
&\leq \lambda \left[\frac{1}{2}(1-|  c_1|^2)+2(1+ \lambda) |c_1|\right]\nonumber\\
&\leq 2\lambda (1+\lambda).
\end{align*}

\medskip\noindent
($ii$) We can easy verify that
\begin{align*}
&|a_5-a^4_2|= \lambda| c_3+2 a_2c_2+\lambda c^2_1 +3a^2_2 c_1|\nonumber\\
&\quad = \lambda| c_3+2 a_2c_2-2\lambda c^2_1 +3c_1(\lambda c_1+a^2_2 )|\nonumber\\
&\quad \leq \lambda(| c_3|+ 2|a_2||c_2|+2 \lambda|c_1|^2+3|c_1||\lambda c_1+a^2_2|)\nonumber\\
&\quad \leq \lambda \bigg[\frac{1}{3}\left(1-|  c_1|^2- \frac{4|c_2|^2}{1+|c_1|} \right)+2(1+ \lambda) |c_2|+2\lambda |c_1|^2+3(1+\lambda+\lambda^2)|c_1| \bigg]\nonumber\\
&\quad =\lambda g_3(|c_1|,|c_2|),
\end{align*}
where
\[g_3(x,y)=\frac{1}{3}\left(1-x^2-  \frac{4y^2}{1+x} \right)+2(1+ \lambda) y+2\lambda x^2+3(1+\lambda+\lambda^2)x,\]
$0\leq x=|c_1|\leq 1,$ $0\leq y=|c_2|\leq\frac{1}{2}(1-x^2).$

\medskip

Since  $\frac{\partial g_3(x,y)}{\partial x}>0,$ $(x,y)\in E,$
then we have that $g_3$ attains its maximum on the boundary of  $E.$

\medskip

For $x=0,$  $ 0\leq y\leq \frac{1}{2}$ we have
\begin{equation*}
g_3(0,y)= \frac{1}{3}(1-4y^2)+2(1+\lambda)y\leq (1+\lambda).
\end{equation*}

\medskip
For $ 0\leq x\leq 1,$ $y=0 $ we have
\begin{align*}
g_3(x,0)&= \left(2 \lambda-\frac{1}{3}\right)x^2+3(1+\lambda+\lambda^2)x+\frac{1}{3}\nonumber\\
& \leq 2 \lambda +3( 1+\lambda+\lambda^2)\nonumber\\
&= 3+5\lambda+3\lambda^2.
\end{align*}

\medskip
Finally, for $ 0\leq x\leq 1,$ $y=\frac{1}{2}(1-x^2) $,
   \begin{equation*}
g_3\left(x,\frac{1}{2}(1-x^2)\right)= 1+\lambda +\left(\frac{1}{3}+3 (1+\lambda+\lambda^2)\right)x-(1-\lambda) x^2-\frac{1}{3}x^3\equiv \varphi_3(x).
\end{equation*}
Since
 \begin{align*}
 \varphi'_3(x)&=\frac{1}{3}+3 (1+\lambda+\lambda^2)-2(1-\lambda) x -x^2\nonumber\\
 & \geq \frac{1}{3}+5\lambda+3\lambda^2> 0,
 \end{align*}
  $\varphi_3$ is an increasing function and
   \begin{equation*}
  \varphi_3(x)\leq \varphi_3(1) = 3+5\lambda+3\lambda^2.
  \end{equation*}

\medskip
  All those facts imply the conclusion of theorem.
\end{proof}

\medskip
\section{Hankel determinant of second and third order}

Let $f\in \A$. Then the $qth$ Hankel determinant of $f$ is defined for $q\geq 1$, and
$n\geq 1$ by
\[
        H_{q}(n) = \left |
        \begin{array}{cccc}
        a_{n} & a_{n+1}& \ldots& a_{n+q-1}\\
        a_{n+1}&a_{n+2}& \ldots& a_{n+q}\\
        \vdots&\vdots&~&\vdots \\
        a_{n+q-1}& a_{n+q}&\ldots&a_{n+2q-2}\\
        \end{array}
        \right |.
\]
Thus, the second and the third Hankel determinants are, respectively,
\[
\begin{split}
H_{2}(2)&= a_2a_4-a_{3}^2,\\
H_{3}(1)&= a_{3}(a_2a_4-a_{3}^2)-a_{4}(a_4-a_{2}a_{3})+a_{5}(a_3-a_{2}^2).
\end{split}
\]

\medskip

In \cite{OP-2019} the authors showed that for the class $\U$, the following estimates are sharp:
$$|H_{2}(2)|\leq1\quad \mbox{and}\quad  |H_{3}(1)|\leq \frac{1}{4}.$$
Here we generalize this result for the classes $\U(\lambda)$.

\medskip

\begin{theorem}\label{thm2.1}
Let $f \in\mathcal{U}(\lambda),$ $0< \lambda \leq 1$, and $f(z)=z+a_2z^2+a_3z^3+\cdots$.  Then
\[ |H_2(2)|\le \frac{\lambda(\lambda+1)}{2} \quad \mbox{and}\quad |H_3(1)|\le \frac{\lambda^2}{4} . \]
The second result is sharp due to the function $f(z)=\frac{z}{1-\lambda/2 z^{3}}=z+\frac{\lambda z^4}{2}+\frac{\lambda^2 z^7}{4}+\cdots$.
\end{theorem}

\begin{proof}
 Using the relation \eqref{eq1.3}, after some calculations  we have:
\begin{align*}
|H_2(2)| &= \left|a_2 a_4-a_3^2 \right| =\lambda \left|a_2  c_2-\lambda c_1^2\right|\\
 &\leq \lambda  \left(|a_2 | |c_2| +\lambda |c_1|^2\right) \\
 &\leq \lambda  \left[(\lambda+1)\left(\frac{1-|c_1|^2}{2} \right) +\lambda  |c_1| ^2\right] \\
 &\leq \lambda  \left[ \frac{\lambda+1}{2}  + \left(\frac{\lambda-1}{2} \right) |c_1| ^2\right]\\
  &\leq \frac{\lambda  (\lambda+1)}{2},
\end{align*}
when $0\leq |c_1| ^2 \leq 1 $, because $\lambda -1\leq0$.

\medskip

In a similar way, using \eqref{eq1.2} and \eqref{eq1.3}, after some calculations, for the third order Hankel determinant we have
 \begin{align*}
 \left|H_3(1) \right|
  &=\lambda^2 |c_1 c_3-c_2^2|\leq \lambda^2  (|c_1| |c_3|+|c_2|^2)\\
 &  \leq \lambda^2 \left[ \frac{|c_1|}{3}\left(1-|c_1|^2-\frac{4|c_2|^2}{1+|c_1|}\right)+|c_2|^2\right]\\
 &= \frac{\lambda^2}{3} \left[  |c_1| -|c_1|^3+\frac{3- |c_1|}{1+|c_1|} \cdot|c_2|^2\right]\\
 &\le \frac{\lambda^2}{3} \left[  |c_1| -|c_1|^3+\frac{3- |c_1|}{1+|c_1|} \cdot\frac{(1-|c_1|^2)^2}{4}\right]\\
  %
%
  &= \frac{\lambda^2}{12} (3-2|c_1|^2-|c_1|^4)\leq\frac{3 \lambda^2}{12}=\frac{  \lambda^2}{4}.
 \end{align*}
Equality is attained for the function $f(z)=\frac{z}{1-(\lambda/2) z^{3}}=z+\frac{\lambda z^4}{2}+\frac{\lambda^2 z^7}{4}+\cdots.$
\end{proof}

\medskip

\end{document}